\documentclass[12pt]{amsart}
\usepackage{amsmath,amscd,amsthm,amsfonts, amssymb,amsxtra, mathrsfs}
\usepackage{calrsfs}
\usepackage[us,12hr]{datetime}
\usepackage[all]{xy} \SelectTips{cm}{}
\usepackage{hyperref}
 \hypersetup{colorlinks=true,citecolor=blue}
   \usepackage{graphicx}
   \usepackage[font=scriptsize]{caption}
   
\usepackage{todonotes}
\usepackage{tikz-cd}

\CDat
\newtheorem{thm}{Theorem}[section]  
\newtheorem*{un-no-thm}{Theorem}
\newtheorem{cor}[thm]{Corollary}     
\newtheorem{lem}[thm]{Lemma}         
\newtheorem{prop}[thm]{Proposition}

\newtheorem{bigthm}{Theorem}
\newtheorem{bigcor}[bigthm]{Corollary}

\newtheorem{bigconj}[bigthm]{Conjecture}

\theoremstyle{definition}
\newtheorem{defn}[thm]{Definition}   

\theoremstyle{definition}

\theoremstyle{definition}
\theoremstyle{remark}
\newtheorem{rem}[thm]{Remark}

\newtheorem{hypo}[thm]{Hypothesis}
\newtheorem{notation}[thm]{Notation}
\newtheorem*{acks}{Acknowledgements}

\newtheorem*{out}{Outline}

\newtheorem*{intro-rem}{Remark}
\newtheorem*{intro-rems}{Remarks}

\newtheorem{ex}[thm]{Example}

\DeclareMathOperator*{\colim}{colim}
\DeclareMathOperator{\id}{id}

\begin{document}
\title[Poincar\'e Complex Diagonals...]{Poincar\'e Complex Diagonals\\ and\\ the Bass trace Conjecture}
\author{John R. Klein} 
\address{Wayne State University,
Detroit, MI 48202} 
\email{klein@math.wayne.edu} 

\author{Florian Naef}
\address{School of Mathematics,
Trinity College, Dublin 2,
Ireland}
\email{naeff@tcd.ie}
\thanks{}
\dedicatory{In memory of Andrew Ranicki}
\begin{abstract} For a finitely dominated Poincar\'e duality space $M$, 
we show how the first author's total obstruction $\mu_M$  to the existence of a Poincar\'e embedding of the diagonal map $M \to M \times M$ in \cite{Klein_diagonals} relates to the Reidemeister trace of the identity map of $M$.  We then apply this relationship to show that $\mu_M$ vanishes when suitable conditions on the fundamental group of $M$ are satisfied.
\end{abstract}

\maketitle
\setlength{\parindent}{15pt}
\setlength{\parskip}{1pt plus 0pt minus 1pt}
\def\smsh{\wedge}
\def\flush{\flushpar}
\def\dbslash{/\!\! /}
\def\:{\colon\!}
\def\Bbb{\mathbb}
\def\bold{\mathbf}
\def\cal{\mathcal}
\def\End{\text{\rm End}}
\def\stableto{\mapstochar \!\!\to}
\def\lr{{\ell r}}
\def\ad{\text{\rm{ad}}}
\def\tr{\sigma}
\def\Top{\text{\rm Top}}

\setcounter{tocdepth}{1}
{
  \hypersetup{linkcolor=olive}
  \tableofcontents
}

\addcontentsline{file}{sec_unit}{entry}

\section{Introduction} \label{sec:intro}

\subsection{Poincar\'e spaces} A distinguishing feature of smooth manifolds is that they carry tangential information: If $M$ is a smooth manifold,
then at each point $x\in M$, one may assign a tangent space $T_xM$.
The tangent spaces assemble into a vector bundle $\tau_M\:TM \to M$. If $M$ is compact, then
the total space $TM$ is diffeomorphic to a tubular neighborhood of the diagonal map $M \to M \times M$.
The existence of a tubular neighborhood corresponds to the fact that an $n$-manifold is locally 
diffeomorphic to $\Bbb R^n$ in a suitably parametrized way.

Note that  a tubular neighborhood determines a decomposition
\[
M \times M \cong D(\tau_M) \cup_{S(\tau_M)} C\,,
\]
where $D(\tau_M)$ is the tangent disk bundle of $M$, $S(\tau_M)$ is the tangent sphere bundle, and $C$ is the complement of the tubular neighborhood.
 
In the 1960s, Browder, Novikov, and Wall developed surgery theory as a tool to classify manifolds up to diffeomorphism in dimensions at least five. 
A  starting point for this theory  is that a manifold satisfies Poincar\'e duality. This led Levitt, Spivak and Wall to introduce
the notion of Poincar\'e duality space \cite{Spivak_duality}, \cite{Wall_Poincare}.  Such spaces satisfy Poincar\'e duality globally, however, in contrast with manifolds,
duality may fail locally.
Nevertheless, Spivak showed that Poincar\'e spaces can be equipped with {\it stable} tangential data: There is a stable spherical fibration, the {\it Spivak tangent fibration},
  which is unique up to contractible choice, and which can be characterized in a homotopy theoretic way.\footnote{Actually, Spivak only defined the {\it normal fibration} of a Poincar\'e space $M$. 
The normal and tangent fibrations are mutual inverses in the Grothendieck group of stable spherical fibrations over $M$.}

One is therefore led to the problem as to what a ``tubular neighborhood'' of the diagonal of a Poincar\'e space should be and 
how it might relate to the Spivak fibration. 

If $M$ is a Poincar\'e duality space formal dimension $n$, then by analogy with the manifold case,
one might define a ``Poincar\'e tubular neighborhood''
of the diagonal $\Delta\: M\to M\times M$ to consist of a homotopy theoretic decomposition
of the form
\[
M\times M \simeq D(\tau) \cup_{S(\tau)} C\, ,
\]
in which $\tau\: S(\tau) \to M$ is an $(n-1)$-spherical fibration which stabilizes to the Spivak tangent fibration,
$D(\tau)$ is the mapping cylinder of $\tau$, and $C$ is some space such that the pair $(C,S(\tau))$ satisfies Poincar\'e duality for pairs.
The above decomposition amounts to  the notion of {\it Poincar\'e embedding} in  the special case of the diagonal map  \cite{Levitt_structure}, \cite[\S11]{Wall_surgery}, \cite{Klein_haef}. 

In this paper we address the following question: When does $M$ possess such a decomposition? 
Assume in what follows that the formal dimension of $M$ is at least four.
In \cite{Klein_compression}, the first author showed that the diagonal admits a Poincar\'e embedding if  $M$ is 1-connected. When $M$ is not $1$-connected
 then in \cite{Klein_diagonals}  the first author defined an obstruction $\mu_M$ whose vanishing is both necessary and sufficient for the existence
of a Poincar\'e embedding of the diagonal. However, he was unable to determine when the obstruction vanishes. 

Our goal is to relate the obstruction $\mu_M$ to other, more well-known ones. We also apply our main results to prove that $\mu_M$ vanishes in some additional cases. 

\subsection{The Wall finiteness obstruction} 
For a group $\pi$, let $K_0(\Bbb Z[\pi])$ be the Grothendieck group of finitely generated projective (left) 
$\Bbb Z[\pi]$-modules.
According to \cite{Wall_finiteness1}, a connected finitely dominated connected space $X$ with fundamental group $\pi$ determines an element
\[
w(X)\in K_0(\Bbb Z[\pi])\, ,
\] 
called the {\it Wall finiteness obstruction}.
If we assume that $\pi$ is finitely presented, then $X$ has the homotopy type of a finite CW complex
if and only if $\tilde w(X)= 0$, where $\tilde w(X) \in  \widetilde{K}_0(\Bbb Z[\pi])$ is the image $w(X)$ in
the reduced Grothendieck group.

\subsection{The Reidemeister characteristic}
Let $\bar \pi$ denote the set of conjugacy classes of $\pi$, and let $\Bbb Z\langle \bar\pi \rangle$ be the free abelian
group with basis $\bar \pi$.
The {\it Hattori-Stallings trace} \cite{Hattori_trace}, \cite{Stallings_trace} is a homomorphism 
\begin{equation} \label{eqn:ht-trace-homo}
r\: K_0(\Bbb Z[\pi]) \to \Bbb Z\langle \bar\pi \rangle
\end{equation}
induced by assigning to a finitely generated projective $\Bbb Z[\pi]$-module $P$ the trace
of the identity map $P \to P$, where the trace is defined by means of the diagram of homomorphisms
\begin{equation}\label{eqn:defn-ht}
\hom_{\Bbb Z[\pi]}(P,P) @< \cong << \hom_{\Bbb Z[\pi]}(P,\Bbb Z[\pi])\otimes_{\Bbb Z[\pi]} P \to \Bbb Z[\pi^\ad]_\pi = \Bbb Z\langle \bar\pi \rangle \, .
\end{equation}
Here, $\hom_{\Bbb Z[\pi]}(P,\Bbb Z[\pi])$ is viewed as a right $\Bbb Z[\pi]$-module and
$\pi^\ad = \pi$ considered as a $\pi$-set together with the conjugation action.
The second displayed homomorphism in \eqref{eqn:defn-ht} is induced by 
evaluation $f\otimes x\mapsto f(x)$.

For a connected finitely dominated space $X$ with fundamental group $\pi$,
the {\it Reidemeister characteristic} 
\[
r(X) \in \Bbb Z\langle \bar\pi \rangle
\]
is defined as $r(X) := r(w(X))$.\footnote{More precisely, $r(X)$ is
 the Reidemeister trace (or generalized Lefschetz trace) of the identity map of $X$.} Let 
$\widetilde{\Bbb Z}\langle \bar\pi \rangle$ denote the cokernel of the
inclusion $\Bbb Z\langle e \rangle \to \Bbb Z\langle \bar\pi \rangle$, where $e \in \pi$ is the trivial element.
The reduced
Reidemeister characteristic $\tilde r(X)$ is the image of $r(X)$ in $\widetilde{\Bbb Z}\langle \bar\pi \rangle$.

Our first main result is:

\begin{bigthm} \label{bigthm:rank-thm} Suppose that $M$ is a connected, finitely dominated Poincar\'e duality space.  
If $\tilde r(M) \ne 0$, then the diagonal map $M \to M^{\times 2}$ does not admit a Poincar\'e embedding.
\end{bigthm}

\begin{rem} There is currently no known example in which $\tilde r(M) \ne 0$. In fact, the {\it Bass trace conjecture} for a group $\pi$
is the statement that the reduced Hattori-Stallings trace 
\[
\tilde r\: {\widetilde K}_0(\Bbb Z[\pi]) \to \widetilde{\Bbb Z}\langle\bar \pi\rangle
\]
is trivial \cite{Bass_Euler}. Although the conjecture remains open, it has been verified for a large class of groups,
including the residually nilpotent ones. See \cite[\S3]{Berrick-Hesselholt} for more details.
\end{rem} 

\subsection{The diagonal obstruction}
Let $d \ge 0$ be an integer.
Then the abelian group ${\Bbb Z}\langle\bar \pi\rangle$ comes equipped with an involution $\alpha \mapsto \bar \alpha$ induced by the operation
$x \mapsto (-1)^{-d}x^{-1}$
for $x\in \pi$.  
We set 
\[
Q_d(\pi) := {\Bbb Z}\langle\bar \pi \rangle_{\Bbb Z_2}\, ,
\] i.e.,  the coinvariants
of the involution. Similarly, one may define a reduced version of 
$Q_d(\pi)$  as the cokernel
\[
\tilde Q_d(\pi) := \text{coker}(Q_d(e) \to Q_d(\pi) )\, ,
\]
where $e$ denotes the trivial group.

The  transfer homomorphism
\[
\tilde{\tr}\: \tilde Q_d(\pi) \to  \widetilde{\Bbb Z}\langle\bar \pi \rangle
\]
is defined by $\tilde{\tr}(\alpha) = \alpha + \bar\alpha$.

Let $M$ be a connected, finitely dominated Poincar\'e duality space of dimension $d$ and fundamental group $\pi$.
Then as alluded to above, in \cite{Klein_diagonals}  an obstruction
\[
\mu_M \in \tilde Q_d(\pi)
\]
was introduced
which vanishes if the diagonal map $M\to M\times M$ admits a Poincar\'e embedding. Conversely, when $d  \ge 4$,
it was shown that $\mu_M = 0$ implies that a Poincar\'e embedding of the diagonal exists. The obstruction $\mu_M$ is
an invariant of the homotopy type of $M$.

Theorem \ref{bigthm:rank-thm} is an immediate consequence of our second main result, which relates the Reidemeister characteristic $r(M)$ to 
the diagonal obstruction $\mu_M$.

\begin{bigthm} \label{bigthm:main} Let $M$ be a connected, finitely dominated Poincar\'e duality space. 
Then
\[
\tilde r(M) = \tilde{\tr}(\mu_M)\, .
\]
In particular, if the diagonal $M \to M^{\times 2}$ admits a Poincar\'e embedding, then $\tilde r(M) = 0$. 
\end{bigthm}

\begin{hypo} \label{hypo:conditions} $M$ is a connected Poincar\'e duality space
of dimension $d$ such that one of the following conditions
hold:
\begin{itemize}
\item $M$ is homotopy finite, or
\item $M$ is finitely dominated and the Bass trace conjecture holds for $\pi = \pi_1(M)$.
\end{itemize}
\end{hypo}

\begin{bigcor}\label{bigcor:main} Assume $M$ satisfies Hypothesis \ref{hypo:conditions}.
Then  $\tilde \tr(\mu_M)$ is trivial.
\end{bigcor}

\begin{rem} As mentioned above, it was known that $\mu_M$ vanishes
when $M$ is simply connected or when $M$ has the homotopy type of a closed
manifold. The only other case that we were aware of
occurs when $M$ is a ``two-stage patch space," i.e., a Poincar\'e space  obtained by gluing
together two compact manifolds along a homotopy equivalence of their boundaries.
If in addition the fundamental group of the boundary is square root closed, then according to 
\cite{Byun} there is a diagonal Poincar\'e embedding, and we infer $\mu_M = 0$ in this instance.
\end{rem} 
Using Corollary \ref{bigcor:main}, we obtain several additional cases:

\begin{bigcor} \label{bigcor:ex} 
Assume that $M$ is a connected, finitely dominated Poincar\'e space and $d = \dim M \ge 4$ is 
even. Additionally, assume that $\pi$ is abelian and 
the 2-primary subgroup of $\pi$ is a finite direct product of cyclic groups of order two.
Then the diagonal map $\Delta \: M \to M^{\times 2}$ admits a Poincar\'e embedding.
\end{bigcor}

\begin{bigcor} \label{bigcor:ex2}
Assume that $M$ satisfies Hypothesis \ref{hypo:conditions} with $d \ge 4$. Suppose
that every non-trivial element of  $\pi$ has odd order.
Then the diagonal map $\Delta \: M \to M^{\times 2}$ admits a Poincar\'e embedding.
\end{bigcor}

The above results justify making the following conjecture.

\begin{bigconj} \label{bigconj:main} Assume that $M$ satisfies Hypothesis \ref{hypo:conditions} with $d \ge 4$. Then
 the diagonal  map $\Delta\: M \to M^{\times 2}$ admits a Poincar\'e embedding.
\end{bigconj}

\begin{out} In \S2 contains the some prerequisites for reading the paper; much of what is contained there
is well-known and in the literature. In \S3 we introduce the Euler invariant of a Poincar\'e space equipped with
an unstable the Spivak tangent fibration. In \S4 we introduce a homotopy-theoretic self-intersection invariant. In \S5 is a discussion
of the equivariant geometric Hopf invariant as in \cite{Crabb-Ranicki} but in the more general setting of a topological group.
The main results are proved in \S6.
\end{out}

\begin{acks} Andrew Ranicki had a profound impact on the first author's mathematical career. He was both a mentor and a friend.  
The current work was inspired by him. The authors wish to thank Ian Leary for pointing out Wall's examples
 \cite[cor.~5.4.2]{Wall_Poincare} of non-homotopy finite Poincar\'e duality spaces.

This research was supported by the U.S.~Department of Energy, Office of Science, under Award Number DE-SC-SC0022134.
\end{acks}

\section{Preliminaries}

Here we develop the minimal foundational scaffolding.
All of this material has appeared elsewhere and the exposition is not to be regarded as exhaustive.

\subsection{Spaces} Let $\Top$ be the Quillen model category of compactly generated weak Hausdorff spaces
\cite{Hirschhorn_model}.\footnote{The reader is to be reminded
that products in $\Top$ are to be retopologized using the compactly generated topology.} The
weak equivalences of $\Top$ are the weak homotopy equivalences, and the fibrations are the Serre fibrations. The cofibrations 
are defined using the right lifting property with respect to the trivial fibrations.
In particular, every object $\Top$ is fibrant. An object is cofibrant whenever it is a retract of a cell complex. We let $\Top_\ast$
denote the category of based spaces. Then $\Top_\ast$ inherits a Quillen model structure from $\Top$ by means of the forgetful functor
$\Top_\ast\to \Top$.

An object in $\Top$ or $\Top_\ast$ is {\it finite} if it is a finite cell complex. It is {\it homotopy finite}
if it is weakly equivalent to a finite object. A object is $X$ is {\it finitely dominated} if it is a retract of a 
homotopy finite object.

We employ the following notation, most of which is standard.
For objects $X,Y\in \Top_\ast$, we let $X\vee Y$ denote the wedge and we let $X\smsh Y$ denote
the smash product. The suspension of $X$ is then $S^1 \smsh X$, where $S^1$ is the circle.   When we write $X^{[k]}$, we mean the 
$k$-fold smash product of $X$ with itself. When passing to the stable category of based spaces we sometimes
write $X\stableto Y$ to indicate a stable map. If $X$ is an unbased space, we write $X_+$ for the based space $X \amalg \ast$ given by taking
the disjoint union with a basepoint. For a based space $A$ and an unbased space $B$, our convention will be $A\smsh B_+ := A \smsh (B_+)$ and
$B_+ \smsh A :=  (B_+) \smsh A$.

\subsection{Poincar\'e duality spaces} Recall that an object $M \in \Top$ is a {\it Poincar\'e duality space} of (formal) dimension $d$ if
there exists a pair
\[
(\cal L,[M])
\]
in which $\cal L$ is a rank one local coefficient system (or {\it orientation sheaf}) and 
\[
[M] \in H_d(M;\cal L)
\]
is a {\it fundamental class} such that for all local coefficient systems $\cal B$, the cap product homomorphism
\[
\cap [M] \: H^\ast(M;\cal B) \to H_{d-\ast}(M;\cal L\otimes\cal B)
\]
is an isomorphism in all degrees (cf.~\cite{Wall_Poincare}, \cite{Klein_Poincare}).  
The Poincar\'e
spaces considered in this paper are assumed to be connected, finitely dominated and cofibrant.
Note that if the pair $(\cal L, [M])$ exists, then it is determined up to unique isomorphism by the above property. 
Also note that closed manifolds
are homotopy finite Poincar\'e duality spaces.

\begin{rem}  Let $p$ be an odd prime.
In \cite[cor.~5.4.2]{Wall_Poincare}, a construction is given of a connected finitely dominated Poincar\'e duality space $K$ of dimension 4
with fundamental group $\pi = \Bbb Z/p\Bbb Z$. For a suitable choice of  $p$, the finiteness obstruction $\tilde w(K) \in \tilde K_0(\Bbb Z[\pi])$
has odd order (for example, $p = 97$ will do, see \cite[p.~31]{Milnor_K-theory}).  Then if $j \ge 3$, 
the product $K\times S^{2j-4}$ is a connected finitely dominated Poincar\'e duality space of dimension $2j$
having non-trivial finiteness obstruction (this uses \cite[thm.~0.1]{Gersten_product}). For other examples see  \cite{Klein_pdfd} .
\end{rem}

More generally, one has the notion of a {\it Poincar\'e pair} $(N,\partial N)$
with fundamental class $[N] \in H_d(N,\partial N; \cal L)$. In this case
one assumes that the cap product
\[
\cap [N] \: H^\ast(N;\cal B) \to H_{d-\ast}(N,\partial N;\cal L\otimes\cal B)
\]
is an isomorphism
and in addition
the image of $[N]$ with respect to the boundary homomorphism
$H_d(N,\partial N; \cal L) \to H_{d-1}(\partial N; \cal L_{|\partial L})$ 
equips $\partial N$ with the structure of a Poincar\'e space of dimension $d-1$.

Given a spherical fibration $\xi\: S(\xi) \to M$, we write $D(\xi) \to M$ for 
the mapping cylinder. The {\it Thom space} is the quotient
\[
M^\xi := D(\xi)/S(\xi)\, .
\]
Note that the first Stiefel-Whitney class of $\xi$ defines an orientation sheaf $\cal L^\xi$. 
Moreover, by the Thom isomorphism, $H_\ast(M^\xi) \cong \tilde H_{\ast-k}(M)$.

A {\it Spivak normal fibration} for a Poincar\'e duality space $M$ of dimension $d$ consists of
a pair $(\nu,c)$ in which $\nu$ is a $(k-1)$-spherical fibration for some $k$
and 
\[
c\: S^{d+k} \to M^\nu
\]
 is a degree one map  in the sense that the image $c_\ast([S^{d+k}]) \in H_{d+k}(M^\xi) \cong H_d(M;\cal L^\xi)$
is a fundamental class for $M$. The pair $(\nu,c)$ always exists and is unique up to stable fiber homotopy equivalence
(see e.g., \cite{Klein_dualizing}). The map $c$ is sometimes called the {\it normal invariant}.

\subsection{Poincar\'e embeddings of the diagonal}
Suppose $M$ is a connected finitely dominated Poincar\'e duality  space of dimension $d$. If $(\cal L,[M])$
is a choice of orientation and fundamental class for $M$, then $(\cal L^{\times 2},[M]^{\times 2})$
is one for $M\times M$. In particular $M\times M$ is a Poincar\'e duality space of dimension $2d$.
 
A {\it diagonal Poincar\'e embedding} of $M$
consists of a map of (finitely dominated) spaces $S(\xi) \to C$ that fits into a commutative homotopy cocartesian square
\[
\xymatrix{
S(\xi) \ar[r] \ar[d]_{\xi} & C \ar[d]\\
M \ar[r]_(.4){\Delta} & M^{\times 2}\,
}
\]
in which 
\begin{itemize}
\item $\Delta$ is the diagonal map,
\item $\xi$ is a $(d-1)$-spherical fibration, 
\item $(C,S(\xi))$ is a Poincar\'e pair of dimension $2d$ with orientation sheaf
given by the restriction $\cal L^{\times 2}_{|C}$ and fundamental
class $[C]$ given by taking the image of $[M]^{\times 2}$ 
under the homomorphism 
\[
H_{2d}(M\times M; \cal L^{\times 2}) \to H_{2d}(M\times M, M; \cal L^{\times 2}) \cong H_{2d}(C,S(\xi);\cal L^{\times 2}_{|C})\, ,
\]
where the displayed isomorphism is given by excision.
\end{itemize}
 It follows that the local system associated to $\xi$ is $\cal L^{-1}$.
We refer the reader to \cite[\S5]{Klein_Poincare} or \cite[\S2]{Klein_diagonals} for the definition of Poincar\'e embeddings in full generality.

\subsection{$G$-spaces}
A topological group object  $G$ of $\Top$ is said to be {\it cofibrant} if its underlying space is. 
Let $\Top(G)$ and $\Top_\ast(G)$ denote the category of left $G$-spaces and based left $G$-spaces.
We remark here that a right $G$-space can always be converted to a left $G$-space using the involution of $G$ defined by 
$g \mapsto g^{-1}$.

The categories $\Top(G)$ and $\Top_\ast(G)$ inherit a model structure using the forgetful
functor to $\Top$. In both instances, all objects are fibrant. An object of  $\Top(G)$ is cofibrant
if it is a retract of a free $G$-cell complex, i.e., a
space built up from the empty space by free $G$-cell attachments, where a free $G$-cell is defined to be $D^n \times G$ for some $n \ge 0$.  
Similarly, an object of $\Top_\ast(G)$ is cofibrant if it is built up from a point by based free $G$-cell attachments.
In $\Top(G)$ and $\Top_\ast(G)$ one can speak of finite, homotopy finite, and finitely dominated objects.

According to Milnor \cite{Milnor_universal}, given a connected based simplicial complex $X$, one may associate a cofibrant topological group $G$ and 
a universal $G$-principal fiber bundle $\tilde X \to X$. 
 As $\tilde X$ is contractible, it follows that $X$ is identified with $BG$ up to homotopy. On the other hand,
any connected space $Y$ is weakly equivalent to a simplicial complex $X$. It follows that $Y$ is weakly equivalent to $BG$.

\begin{rem} In this case, one can identify the homotopy category of $\Top(G)$ with $\Top/X$, that is spaces over $X \simeq BG$. Concretely, a $G$-space $E$ is sent to $E_{hG} \to \star_{hG} = BG$, whereas to a map $f \colon Z \to X$ we assign the $G$-space $\tilde X \times_X Z$ which models the homotopy fiber of $f$ as a $G$-space.
\end{rem}

If $X, Y\in \Top_\ast(G)$ are objects, we write
\[
[X,Y]_G
\]
for the set of homotopy classes of based $G$-maps $X\to Y$.  Similarly, we write
\[
\{X,Y\}_G
\]
for the abelian group of homotopy classes of stable $G$-maps $X\stableto  Y$.

\subsection{Naive $G$-spectra} 
We will also be considering the category $\text{\rm Sp}(G)$ of (naive) $G$-spectra formed from objects of $\Top_\ast(G)$.
A $G$-spectrum $E$ in this sense consists of based $G$-spaces $E_n$ for $n \ge 0$ together with based $G$-maps
$\Sigma E_n \to E_{n+1}$ (structure maps), where $G$ acts trivially on the suspension coordinate of $\Sigma E_n$. 
A morphism $f\: E\to E'$ of $G$-spectra consists of based $G$-maps of $f_n \: E_n \to E_n'$ for all $n$ which
are compatible with the structure maps. 

Then $\text{\rm Sp}(G)$ forms a model category in which 
the fibrant objects are the $\Omega$-spectra and the weak equivalences are the maps which induce
isomorphisms on homotopy groups after applying fibrant replacement \cite{Schwede_model}.  

\begin{notation} We indicate
a weak equivalence $f\: E\to E'$ with the decoration $@> {}_\sim >>$.
If there is a finite chain of weak equivalences connecting $E$ to $E'$, then by 
 slight notational abuse, we write $E \simeq E'$. 
 \end{notation}

 Let $S$ denote the sphere spectrum.
Given any $G$-space $X$ its suspension spectrum $\Sigma^\infty (X_+):= S \smsh X_+$, is naturally a $G$-spectrum. We will sometimes denote it by
\[
S[X]\, .
\]

\begin{rem} It is sometimes convenient to think of a $G$-spectrum as a collection of based $G$-spaces $E_V$ 
indexed over finite dimensional inner product
spaces $V$, together with $G$-maps $S^W \smsh E_V \to E_{V\oplus W}$.
\end{rem}

Given a based $G$-space $X$ and a  $G$-spectrum $E$, we may form the spectrum of equivariant maps from $X$ to $E$:
\[
F_G(X,E)\, .
\]
The $n$-th space of $F_G(X,E)$ is given by the function space $F_G(X,E_n)$ consisting of based $G$-maps  
$X\to E_n$.   For $F_G(X,E)$ to have a sensible homotopy type, one should assume that $X$ is cofibrant and that $E$ is fibrant.
Similarly, one may form
\[
X \smsh_G E\, ,
\]
which is the spectrum having $n$-space $X\smsh_G E_n$, i.e., the orbits of $G$ acting diagonally on the smash product $X\smsh E_n$.
For this construction to have a sensible homotopy type, we should assume that $X$ and $E$ are cofibrant.
We will implicitly apply fibrant/cofibrant replacement to arrange that the above constructions are always homotopy invariant.

As a special case of the above,
let $EG$ be a cofibrant free contractible $G$-space. Then the {\it homotopy fixed points} and {\it homotopy orbits} of $G$ acting on $E$ are given by
\[
E^{hG} := F_G(EG_+,E) \quad \text{and} \quad E_{hG} := E\smsh_G EG_+\, .
\]

\subsubsection{Function spectra and smash products}  
By resorting to one of the standard enriched model categories of spectra \cite{Smith_symmetric},\cite{EKMM}, we can form internal hom objects
\[
\hom(X,Y)
\]
for naive spectra $X$ and $Y$. 

Note that $X$ should be cofibrant and $Y$ should be fibrant for the function spectrum to be homotopy
invariant. If $X$ and $Y$ are naive $G$-spectra in one of these categories, then $\hom(X,Y)$ is equipped with a $(G\times G)$-action
and restricting to the $G$-equivariant functions gives the spectrum
\[
\hom_G(X,Y)\, .
\]
When $X = \Sigma^\infty Z$ is a suspension spectrum of a (cofibrant) $G$-space, then one  has an identification
\[
\hom_G(X,Y) \simeq F_G(Z,Y)\, .
\]
Similarly, the smash product $X\smsh Y$ has the structure of a $(G\times G)$-spectrum and the orbits under the diagonal action
of $G$ defines a spectrum $X\smsh_G Y$. Moreover, $X\smsh_G Y \simeq Z\smsh_G Y$ when $X= \Sigma^\infty Z$.

\subsection{The dualizing spectrum} The main reference
for this subsection is \cite{Klein_dualizing}. 
Let $G$ be a cofibrant topological group.
Let $G^\lr$ denote $G$ with the left action of $G\times G$ given by $(g,h)\cdot x = gxh^{-1}$.


\begin{rem} The $(G \times G)$-spectrum $S[G^\lr] $ has the following useful property: given any $G$-spectrum $E$, let $E^\ell$ and $E^r$ denote $E$ as a $G \times G$-spectrum with trivial right/left action, respectively. Then one has a natural isomorphism of $(G \times G)$-spectra
\[
E^\ell \smsh S[G^\lr] \cong E^r \smsh S[G^\lr]\, .
\]
If we write $G_\ell$ in place of $G\times 1$ and $G_r$ in place of $1\times G$, then
we immediately infer that
\[
(E^\ell \smsh S[G^\lr])_{hG_\ell} \simeq  E \simeq (E^r \smsh S[G^\lr])_{hG_r}
\]
as $G$-spectra and compatible with the above identification. Similarly, if $E$ is a $(G \times G)$-spectrum, we obtain a $G$-spectrum $E^{\ad}$ by restricting to the diagonal action. We then obtain natural equivalences
\begin{equation}\label{eqn:untwist}
(E \smsh S[G^\lr])_{hG_\ell} \simeq E^{\ad} \simeq (E \smsh S[G^\lr])_{hG_r}.
\end{equation}
\end{rem}

\begin{ex}  Let $G^\ad$ denote $G$ equipped with its conjugation action $g\cdot x= gxg^{-1}$. We set
\[
S[G^\ad] := S \smsh (G^\ad_+)\, .
\]
Then $S[G^\ad] = S[G^{\lr}]^{\ad}$.
\end{ex}

The {\it dualizing spectrum} of $G$
is the $G$-spectrum
\begin{equation} \label{eqn:dualizing}
\cal D_G := S[G^\lr]^{hG_\ell} = F_{G_\ell}(EG_+,S[G^\lr])\, ,
\end{equation}
in which $G$ acts as $G_r$.

For any $G$-spectrum $E$, one has a {\it norm map}
\begin{equation} \label{eqn:norm-map}
\eta = \eta_E \: \cal D_G \smsh_{hG} E \to E^{hG}
\end{equation}
which is natural in $E$.  The map $\eta$ arises by noting that composition of functions
gives rise to a pairing
\begin{equation} \label{eqn:norm-compostion}
F_{G_\ell}(EG_+,S[G^\lr]) \smsh \hom_{G_\ell}(S[G^\lr],E) \to F_{G_\ell}(EG_+,E) 
\end{equation}
which is $G$-invariant, and therefore factors through homotopy orbits. 
The norm map is then obtained using the evident weak equivalence $E \simeq \hom_{G_\ell}(S[G^\lr],E)$. We recall the following
\begin{thm}[Theorem D in \cite{Klein_dualizing}]\label{them:nmequiv}
	If $BG$ is finitely dominated, then the norm map
	\[
	\eta_E \: \cal D_G \smsh_{hG} E = (\cal D_G \smsh E)_{hG} \to E^{hG}
	\]
	is an equivalence for any $G$-spectrum $E$.
\end{thm}


Note that the norm map is adjoint to a map of $G$-spectra
\[
\cal D_G \smsh_{hG} E \smsh EG_+ \to E,
\]
This map can be obtained from the evaluation map
\begin{equation}\label{eqn:epsilon}
\epsilon\: \cal D_G\smsh EG_+ \to S[G^\lr]
\end{equation}
which is $(G\times G)$-equivariant.  The map $\epsilon$ is defined so that it is
adjoint to the identity map of $\cal D_G$, i.e., it is given by the evaluating
stable maps $EG_+ \to S[G^\lr]$ at points of $EG_+$.

\begin{ex} In the case $E = S$, with $G$ acting trivially, the norm equivalence is
\[
\eta \: \left( \cal D_G \right)_{hG} \xrightarrow{\simeq} S^{hG} = F(BG_+,S)\, .
\]
Hence, the unit $BG_+ \to S$ 
corresponds to  a map $S \to \left( \cal D_G \right)_{hG}$. 

This relates to the classical normal invariant of the Spivak normal fibration as follows:
Suppose that $M$ is a (connnected, finitely dominated) Poincar\'e duality space of dimension $d$.  Let $\tilde M \to M$ be a universal principal
bundle over $M$ with topological structure group $G$. Then $\tilde M$ is a free, contractible $G$-space.
We may assume that $G$ is cofibrant. Then $\tilde M$  models $EG$ and $M$ models $BG$. In this case
$\cal D_G$ is unequivariantly weakly equivalent to a sphere of dimension $-d$, and we write 
\[
S^{-\tau} := \cal D_G
\]
in this instance. Then the stable spherical fibration
\[
S^{-\tau} \times_G \tilde M \to M
\]
is the Spivak normal fibration of $M$.\footnote{Here we are considering the Spivak normal fibration as a parametrized spectrum over $M$.} In this context, the Thom spectrum $M^{-\tau}$ is identified with $S^{-\tau}_{hG} = S^{-\tau} \smsh_G (\tilde M_+)$, and
the norm equivalence in this notation is just Atiyah duality \cite{Atiyah_thom} for Poincar\'e duality spaces:
\[
M^{-\tau} \simeq F(M_+,S)\, .
\]
Summarizing, with respect to $E = S$, the norm equivalence specializes to Atiyah duality and
 the stable normal invariant $S \to M^{-\tau}$ corresponds to the unit $1\: M_+ \to S$. 
 \end{ex}
 
In what follows we set
\[
S^\tau := \hom(S^{-\tau},S)\smsh EG_+\, .
\]
Then  $G=G_\ell$ acts on $S^\tau$ and 
\[
S^{\tau} \times_G \tilde M \to M
\]
is the stable {\it Spivak tangent fibration} of $M$. In what follows we extend the action
of $G = G_\ell$ on $S^\tau$ to $G\times G$ by letting $G_r$ act trivially.

\subsection{The umkehr map $\Delta_!$}
To the diagonal map $\Delta \: M \to M\times M$ we
associate a $(G\times G$)-equivariant {\it umkehr map}
\begin{equation} \label{eqn:umkehr}
\Delta_!\: EG_+ \smsh EG_+ \to S^\tau[G^\lr]\, ,
\end{equation}
where $S^\tau[G^\lr] := S^\tau\smsh G^\lr_+$.  
The umkehr map  is  defined as the evaluation map
\[
EG_+ \smsh F_{G_\ell}(EG_+,S^\tau[G^\lr])  \to S^\tau[G^\lr]\, ,
\]
in conjunction with the observation that the norm map with respect to $E= S^\tau[G^\lr]$ defines an equivalence
\[
EG_+ \simeq \cal D_G \smsh_{hG} S^\tau[G^\lr] @> \eta > ^\sim  >
F_{G_\ell}(EG_+,S^\tau[G^\lr]) \, .
\]
\begin{rem} If $M$ is a closed manifold, then the homotopy orbits
of $G\times G$ acting on $\Delta_!$ coincides with the map $M_+\smsh M_+ \to M^\tau$ given
 Pontryagin-Thom collapse of the embedding $\Delta$.
\end{rem}

We can now rephrase the norm equivalence in terms of $\Delta_!$. Given any $G$-spectrum $E$ we obtain a map of $G$-spectra
\[
EG_+ \smsh E_{hG} \cong (EG_+ \smsh EG_+ \smsh E^r)_{G_r} @> \Delta_! \smsh \id_E >> (S^\tau[G^\lr] \smsh E^r)_{G_r} \cong S^\tau \smsh E \, ,
\]
The proof of the following is straightforward, but tedious.  We leave its verification  to the reader.
\begin{lem}
The map adjoint to $\Delta_! \smsh \id_E$ defines a weak equivalence
\[
E_{hG} @> \sim >> (S^\tau \wedge E)^{hG}
\]
that is natural in $E$.
\end{lem}

\section{The tangential Euler invariant}

Assume that $M$ is a finitely dominated Poincar\'e duality space of dimension $d$.
Suppose in addition that an ({\it un}stable) $(d-1)$-spherical fibration 
\[
\xi\: S(\xi) \to M
\]
is specified.
We also assume that $\xi$ comes equipped with a choice of stable fiber homotopy equivalence to the Spivak tangent fibration of $M$.

Let $D(\xi) \to M$ be the mapping cylinder of $\xi$. For any space $Y \to M$, let $\tilde Y := \tilde M \times_M Y$. Then $\tilde Y$ 
is an unbased $G$-space. Define
\[
S^\xi := \tilde D(\xi)/\tilde S(\xi)\, .
\]
Then $S^\xi$ is a based $G$-space which is unequivariantly the homotopy type of a $d$-sphere (note: $S^\xi$ models
the fiber of the unreduced fiberwise suspension of $\xi$).  

In particular, one has a diagonal map 
\[
\Delta_\xi\: S^\xi[G^\lr] \to S^\xi[G^\lr]\smsh S^\xi[G^\lr]\, ,
\]
which depends on $\xi$. In view of the stable identification $S^\xi \simeq S^\tau$, we resort to the notation
\[
\Delta_\xi\: S^\tau[G^\lr] \to S^\tau[G^\lr]\smsh S^\tau[G^\lr]
\]
when $\Delta_\xi$ is considered as a stable map.

\begin{defn} The {\it Euler invariant} $ \mathfrak e(\xi)$ is the $(G\times G)$-equivariant composition
\[
EG_+ \smsh EG_+ @> \Delta_! >> S^\tau[G^\lr] @> \Delta_\xi >> S^\tau[G^\lr]\smsh S^\tau[G^\lr]
\]
\end{defn}

\begin{lem}\label{lem:free-loop-ident} There is a preferred isomorphism of abelian groups
\[
 \{EG_+ \smsh EG_+, S^\tau[G^\lr]\smsh S^\tau[G^\lr]\}_{G\times G} \cong H_0(LM)\, ,
 \]
 where $LM := \text{\rm map}(S^1,M)$ is the free loop space of $M$.
 \end{lem}
 
 \begin{rem} \label{rem:free-loop-ident} The group $H_0(LM)$ is canonically isomorphic to
$\Bbb Z\langle\bar \pi\rangle$.
\end{rem}

 \begin{proof}[Proof of Lemma \ref{lem:free-loop-ident}]  The norm map applied to the $(G\times G)$-spectrum $S^\tau[G^\lr]\smsh S^\tau[G^\lr]$
 defines an equivalence
\small \begin{equation} \label{eqn:reduction-free-loop}
(S^{-\tau}{\smsh} S^{-\tau}) \smsh_{h G^{\times 2}} (S^\tau[G^\lr]\smsh S^\tau[G^\lr]) \simeq  (S^\tau[G^\lr]\smsh S^\tau[G^\lr])^{hG^{\times 2}}\, .
 \end{equation}
 \normalsize
As $S\simeq_{G} S^{-\tau} \smsh S^\tau$,  the left side of \eqref{eqn:reduction-free-loop} coincides up to homotopy with
\[
S[G^\lr] \smsh_{hG^{\times 2}} S[G^\lr] \simeq (S[G^\lr \times G^\lr])_{hG^{\times 2}}\, .
\]
Consider the map $G^\lr \to G^\lr \times G^\lr$ given by the inclusion of $1\times G^\lr$. It is straightforward to check that the induced map
of spectra
\[
S[G^\lr]_{hG} \to S[G^\lr] \smsh_{hG^{\times 2}} S[G^\lr] 
\]
is a weak equivalence. But it is well known that $S[G^\lr]_{hG}$ coincides with $\Sigma^\infty (LM)_+$ (see e.g., \cite[lem.~9.1]{Klein-Schochet-Smith}). 
Hence, we have defined a weak equivalence of spectra
\[
(S^\tau[G^\lr]\smsh S^\tau[G^\lr])^{hG^{\times 2}} \simeq \Sigma^\infty (LM)_+ \, .
\]
The conclusion of the Lemma follows by taking $\pi_0$.
\end{proof}

In view of Lemma \ref{lem:free-loop-ident}, we infer that
\[\mathfrak e(\xi) \in H_0(LM) = \Bbb Z\langle \bar\pi \rangle\, .
\]

\begin{prop} \label{prop:free-loop-ident}
For any $\xi$ as above, $\mathfrak e(\xi)$ lies in the summand defined by the conjugacy class of the trivial 
element of $\pi$.
\end{prop}

\begin{rem}  The proposition is equivalent to the assertion that
$\mathfrak e(\xi) \in H_0(LM)$ factors through $H_0(M)$. Alternatively,
the image of $\mathfrak e(\xi)$ in $H_0(LM,M)$ is trivial.
\end{rem}

\begin{proof} Consider the homomorphism
\[
\phi\:\{S^\tau,S^\tau\smsh S^\tau\}_{G} \to \{EG_+\smsh EG_+,S^\tau[G^\lr]\smsh S^\tau[G^\lr]\}_{G\times G} 
\]
induced by  assigning to a stable $G$-map $f\: S^\tau \to S^\tau\smsh S^\tau$
the stable $(G\times G)$-map
\[
EG_+ \smsh EG_+ @> \Delta_! >> S^\tau \smsh G_+^\lr @> f \smsh \Delta_{G_+^\lr} >> S^\tau\smsh S^\tau \smsh G_+^{\lr} \smsh G_+^{\lr}\, .
\]
Then $\Delta_\xi = \phi(\delta_\xi)$, where $\delta_\xi\: S^\xi\to S^\xi \smsh S^\xi$ 
is the diagonal map of $S^\xi$. Furthermore, smashing with $S^\tau$ defines  an isomorphism
\[
\{S,S^\tau\}_{G} \cong \{S^\tau,S^\tau\smsh S^\tau\}_{G}\, .
\]
On the other hand, the norm map applied to
$S^\tau$ defines an equivalence 
\[
\Sigma^\infty M_+ \simeq S_{hG} \simeq \cal D_G \smsh_{hG} S^{\tau} @> \eta > ^\sim> (S^{\tau})^{hG}
\]
where in the above $G$ acts trivially on $S$ and recall that $\cal D_G \simeq S^{-\tau}$.
Taking $\pi_0$, we obtain an
isomorphism 
\[
\{S,S^\tau\}_{G} \cong \pi^{\text{st}}_0(M) = \Bbb Z\, .
\]
With respect the the identifications, the homomorphism  $\phi$
corresponds to inclusion $\Bbb Z\langle e\rangle \to \Bbb Z\langle\bar \pi \rangle$.
\end{proof}

\begin{rem} With slightly more effort, it is possible to show that 
\[
\mathfrak e(\xi) = e(\xi) \cap [M] \in H_0(M)\, ,
\]
where $e(\xi)  \in H^d(M, \cal L^{-1})$  is the  (twisted) Euler class of the spherical fibration
$\xi$.
\end{rem}

\section{Self-intersection}
 \label{sec:cup}

\begin{defn} Let $M$ be a finitely dominated Poincar\'e duality space as above.
The {\it self-intersection invariant}  of  $M$ is the 
 homology class
\[
I(M) \in H_0(LM)
\]
given by the $G$-equivariant stable homotopy class of the composition
\[
(EG_+)^{[2]} @> \Delta >> (EG_+)^{[2]} \smsh (EG_+)^{[2]} @> \Delta_! \smsh \Delta_!>> S^\tau[G^\lr] \smsh S^\tau[G^\lr]\, ,
\]
where we use the identification of Lemma \ref{lem:free-loop-ident}.
\end{defn}

\subsection{The Reidemeister characteristic} The {\it Reidemeister characteristic} of $M$ 
is the homology class
\[
r(M) \in H_0(LM)
\] 
which is given by the stable homotopy
class of the composition of
maps of the form
\[
S \to M^{-\tau} \stableto (LM)_+  \, .
\]
The first map in this composition is the stable normal invariant
and the second map is defined as follows:
recall that the evaluation map 
\[
\epsilon\: S^{-\tau} \smsh EG_+ \to S[G^\lr]
\]
is $(G\times G)$-equivariant.
We take homotopy orbits of $\epsilon$ with respect to the diagonal subgroup  $G \subset G\times G$ to obtain the map
\[
M^{-\tau} \simeq S^{-\tau}_{hG} \to S[G^\ad]_{hG} \simeq S\smsh ((LM)_+)\, .
\]
(compare \cite{Berrick-Hesselholt}, \cite{Berrick-idempotents}).

\begin{rem} As $M$ is finitely dominated, there is a homotopy finite space $X$ and a factorization of the identity map
\[
M @> s >> X @>r >> M\, .
\]
Setting $f = s\circ r\: X \to X$, it follows that $f^{\circ 2} = f$, so $f$ is idempotent.  
Moreover, if $\pi = \pi_1(M)$, and $u \: M \to B\pi$ classifies the universal cover, then one has the composition
\[
v \: X @> r >> M @> u >> B\pi\, .
\]
Consequently, $v \circ f = v$, i.e., $f$ is a map over $B\pi$. Then the generalized Lefschetz trace $L(f) \in \Bbb Z\langle\bar \pi\rangle = H_0(LM)$ is defined 
(cf.\cite[\S4]{Berrick-idempotents}, \cite{Geoghegan_Nielsen}).
It is not hard to show that 
\[
r(M) = L(f)\, .
\] However, we will not make use of this fact.
\end{rem}

We next consider the norm equivalence
\[
\eta\: S[G^\ad]_{hG} @> \sim >> S^\tau[G^\ad]^{hG} := F_G(EG_+, S^\tau[G^\ad])
\]
for the $G$-spectrum $S^\tau[G^\ad]$. Recall that $S^\tau[G^\ad]$ is $S^\tau[G^{\lr}]$
when considered as a $G$-spectrum.

The following result is essentially an unravelling of the definitions. We omit the proof.

\begin{lem}\label{lem:Rtraceiseuler}
With respect to the above norm equivalence,
the Reidemeister characteristic $r(M)$ is represented by the $G$-equivariant stable composite
\[
EG_+ @> \Delta >> EG_+ \smsh EG_+ @> \Delta_! >> S^\tau[G^{\lr}]\, .
\]
\end{lem}

Let $E$ be any $(G\times G)$-spectrum. smashing $\Delta_!$ with $E$ and taking homotopy fixed points, we obtain a map
\[
\Delta_! ^E \: E^{h(G \times G)} \to (E\smsh S^{\tau}[G^\lr])^{h(G \times G)}\, .
\]
In the above, we have implicitly used the fact that the diagonal map 
 $\Delta_{EG_+ \smsh EG_+} \colon EG_+ \smsh EG_+ \to (EG_+ \smsh EG_+)^{[2]}$  is a
 $(G\times G)$-equivariant weak equivalence.
 
If we restrict $(G \times G)$-fixed points to $G$-fixed points along the diagonal inclusion $\Delta \colon G \to G \times G$, we  obtain a map
\[
\Delta_E^* \: E^{h(G \times G)} \to  (E^{\ad})^{hG}\,  .
\]
The following lemma in effect says that the maps $\Delta_! ^E $ and $\Delta_E^*$ coincide under the norm equivalence:

\begin{lem}\label{lem:triangle}
Let $E$ be any cofibrant $(G{ \times} G)$-spectrum. Then the following triangle commutes up to homotopy
\[
\begin{tikzcd}
E^{h(G \times G)} \ar[r, "\Delta_!^E"] \ar[rd, "\Delta^*_E" '] & (E \smsh S^\tau[G^\lr])^{h(G \times G)} \\
& (E^{\ad})^{hG} \ar[u, "\simeq" '] ,
\end{tikzcd}
\]
in which the vertical arrow is given by
\[
(E^{\ad})^{hG} \simeq ((E \smsh S[G^\lr])_{hG_\ell})^{hG} @> \eta ^{hG}>^\sim > ((E \smsh S^\tau[G^\lr])^{hG_\ell})^{hG} \cong (E \smsh S^\tau[G^\lr])^{h(G \times G)} \, ,
\]
where $\eta$  is the norm equivalence applied to $E \smsh S^\tau[G^\lr]$.
\end{lem}

\begin{proof}
Note that the diagram is natural in $E$ and the functor $E\mapsto E^{h(G \times G)}$ is represented by 
$(EG_+)^{[2]} = EG_+ \smsh EG_+$, where to avoid clutter, we have identified $EG_+$ with $S\smsh EG_+$.
 
Consequently, by the Yoneda lemma, it suffices to show that the two composites
\[
\xymatrix{
S    \ar[r]^(.2)u & ((EG_+)^{[2]})^{h(G \times G)}  \ar[rr]^{\Delta_!^{EG_+^{[2]}}} \ar[drr]_{\Delta^\ast_{EG_+^{[2]}}}  &&  ((EG_+)^{[2]} \smsh S^\tau[G^\lr])^{h(G \times G)}  \\
& & &  ((EG_+^{[2]})^{\ad})^{hG}  . \ar[u]_\simeq
} 
\]
are homotopic. Using that $EG_+$ is the unit we can simplify the diagram to
\[
\begin{tikzcd}
	S \ar[r, "u"] & (EG_+ \smsh EG_+)^{h(G \times G)} \ar[r] \ar[rd] & (S^\tau[G^\lr])^{h(G \times G)} \\
	&& (EG_+)^{hG} \ar[u, "\simeq" ']\, ,
\end{tikzcd}
\]
where the  upper composite is adjoint to $\Delta_!$. Note that the unit $S \to (EG_+ \smsh EG_+)^{h(G \times G)}$ is mapped to the unit in $(EG_+)^{hG}$ under the diagonal map. Thus the lower composite is the unit $S \to (EG_+)^{hG}$ composed with the equivalence
\begin{multline*}
(EG_+)^{hG} \simeq ((EG_+ \smsh S[G^\lr])_{hG_\ell})^{hG} @>\eta^{hG}  >{}^\sim> ((EG_+ \smsh S^\tau[G^\lr])^{hG_\ell})^{hG} \cong \\ 
\cong (EG_+ \smsh S^\tau[G^\lr])^{h(G \times G)} \cong (S^\tau[G^\lr])^{h(G \times G)}\, .
\end{multline*}
Observe that the last equivalence is obtained from the norm equivalence
\[
\eta \: EG_+ \overset{\simeq}{\to} (S^\tau[G^\lr])^{hG_\ell}
\]
by passing to $G$-homotopy fixed points. Thus the lower composite is adjoint to the $(G \times G)$-equivariant map
\[
EG_+ \smsh EG_+ @> \id\smsh \eta >>  EG_+ \smsh (S^\tau[G^\lr])^{hG_\ell} \to S^\tau[G^\lr],
\]
which was the definition of $\Delta_!$ (here, the second displayed map is given by evaluation).
\end{proof}

%

\begin{prop}\label{prop:I=r} For any finitely dominated Poincar\'e duality space $M$, we have $I(M) = r(M)$.
\end{prop}

\begin{proof}  
By definition $I(M)$, is the image under $\Delta_! \smsh -$ of $\Delta_!$ itself, where $E = S^\tau \smsh S[G^\lr]$. Lemma \ref{lem:triangle} provides an identification with $r(M)$ in the form of Lemma \ref{lem:Rtraceiseuler} in the previous lemma.
\end{proof}

\begin{cor} If $M$ is homotopy finite, then $I(M) \in H_0(LM)$ lies in summand defined by the
conjugacy class of the trivial element of $\pi$.
\end{cor}

\begin{proof} 
By the proposition it is enough to show that $r(M)$ lies in the trivial conjugacy  class.

It will be convenient to invoke Waldhausen's algebraic $K$-theory of spaces functor $X\mapsto A(X)$.
Here $A(X)$ is the $K$-theory of the category  with cofibrations and weak equivalences
given by the retractive spaces over $X$ which are relatively finitely dominated \cite{Wald_1126}.

As $M$ is finitely dominated, $M\times S^0$ is
a relatively finitely dominated retractive space over $M$ and therefore determines an element
of $\pi_0(A(M)) \cong K_0(\Bbb Z[\pi_1(M)])$. With respect to this identification,
this element is  Wall's finiteness obstruction $w(M)$. 
The image of $w(M)$ in $H_0(LM)$ under the B\"okstedt-Dennis trace map \cite{Boekstedt_THH}
is identified with $r(M)$. 

If $M$ is homotopy finite, then $w(M)$ lies in the image of the homomorphism $\pi_0(A(\ast)) \to \pi_0(A(M))$. 
But the composition $\pi_0(A(\ast)) \to  \pi_0(A(M)) \to H_0(LM)$ factors through $H_0(L\ast) =\Bbb Z$. 
\end{proof}

\begin{cor}  Assume that $M$ is finitely dominated. Suppose that the Bass trace conjecture holds for the group $\pi$.
Then the self-intersection invariant 
$I(M) \in \Bbb Z\langle\bar \pi\rangle$ lies in the summand defined by the  conjugacy class of the trivial element of $\pi$.
\end{cor}

\section{Hopf invariants}

The main reference for this section is the book of Crabb and Ranicki \cite{Crabb-Ranicki}.

Let $Y$ be a cofibrant based $G$-space. Then $Y^{[2]}= Y\smsh Y$ is a based $(\Bbb Z_2 \times G)$-space.
We may then form the (not naive) $(\Bbb Z_2 \times G)$-spectrum 
\[
\Sigma^\infty_{\Bbb Z_2} (Y^{[2]})\, .
\]
indexed over the complete $\Bbb Z_2$-universe $\cal U$ (i.e., a countable direct sum of copies of the
regular representation $\Bbb R[\Bbb Z_2]$).

Concretely, the zeroth space of $\Sigma^\infty_{\Bbb Z_2} (Y^{[2]})$  is given by the colimit
\[
\colim_{V \in \cal U} \Omega^VS^V (Y\smsh Y) = \colim_{V\in \cal U} F(S^V,S^V\smsh Y^{[2]})
\]
where $S^V$ is the one-point compactification of $V$
and  $F(S^V,S^V\smsh Y^{[2]})$ denotes the function space of unequivariant based maps $S^V\to S^V \smsh Y^{[2]}$.

According to \cite[\S V.11]{LMS_equivariant}, one has a tom~Dieck fiber sequence of spectra with $G$-action
\begin{equation} \label{eqn:tomDieck-spectra}
\Sigma^\infty D_2(Y) @> \iota >> (\Sigma^\infty_{\Bbb Z_2} (Y^{[2]}))^{\Bbb Z_2} @> \phi >> \Sigma^\infty Y\, .
\end{equation}
in which the the middle term of \eqref{eqn:tomDieck-spectra} is the categorical $\Bbb Z_2$-fixed point spectrum and
$D_2(Y) = Y^{[2]}_{h\Bbb Z_2}$ is the quadratic construction. 

The zeroth space of the middle term of \eqref{eqn:tomDieck-spectra} will be denoted by
\[
Q_{\Bbb Z_2} (Y^{[2]})^{\Bbb Z_2}\, .
\]
Its points are represented by $\Bbb Z_2$-equivariant maps $\alpha\: S^V\to S^V \smsh Y^{[2]}$.
The zeroth space of the rightmost term of \eqref{eqn:tomDieck-spectra} is 
 $QY = \Omega^\infty \Sigma^\infty (Y)$; its points are represented by a map $\beta\: S^W\to S^W\smsh Y$
for some finite dimensional inner product space  $W$.

The displayed map  $\phi$ in \eqref{eqn:tomDieck-spectra} assigns to such an $\alpha$ the induced map of fixed point spaces
\[
\alpha^{\Bbb Z_2}\: S^{V^{\Bbb Z_2}} \to S^{V^{\Bbb Z_2}} \smsh Y
\]
where we have used the fact that $Y$ is the fixed point space of $\Bbb Z_2$ acting on $Y^{[2]}$.
Consequently, on the level of zeroth spaces, there is a homotopy fiber sequence of based $G$-spaces
\begin{equation} \label{eqn:tomDieck}
QD_2(Y) @> \iota >> Q_{\Bbb Z_2} (Y^{[2]})^{\Bbb Z_2}@> \phi >> QY
\end{equation}
which is $G$-equivariantly split.

The splitting of \eqref{eqn:tomDieck}
may be defined as follows: a point of $QY$ is represented by a map
$\beta\: S^W \to S^W\smsh Y$ in which $W$ is a trivial $\Bbb Z_2$-representation (i.e., a finite dimensional
inner product space). This induces  the $\Bbb Z_2$-equivariant map
\[
S^W @> \beta >> S^W \smsh Y @>1_W\smsh \Delta_Y >> S^W\smsh Y^{[2]} 
\]
representing a point of $Q_{\Bbb Z_2} (Y^{[2]})^{\Bbb Z_2}$. This provides a section to 
the map $\phi$ and defines the splitting.

Let $X \in \Top_\ast(G)$ be a finite object.
Then a stable $G$-map $f\: X \stableto Y$, is represented by a $G$-map $f_W\: S^W\smsh X \to S^W\smsh Y$.
One may associate to $f_W$ the $(\Bbb Z_2\times G)$-equivariant map
\[
\Delta_Y \circ f_W\: S^W\smsh X @> f_W >>  S^W\smsh Y @>1 \smsh \Delta_Y >> S^W\smsh Y^{[2]}\, .
\]
One may also associate to $f_W$ the $(\Bbb Z_2\times G)$-equivariant map
\[
(f_W\smsh f_W)\circ \Delta_X\:  S^{W\oplus W} \smsh X @>1 \smsh \Delta_X >>  S^{W\oplus W}\smsh X^{[2]}  @> f_W\smsh f_W >>   S^{W\oplus W}\smsh Y^{[2]}\, ,
\]
where it is understood that $W\oplus W$ is given the transposition action of $\Bbb Z_2$, and we have implicitly shuffled the factors of the displayed smash product.
It follows that  $\Delta_Y \circ f_W$ and $(f_W\smsh f_W)\circ \Delta_X$ represent a pair of $G$-maps
\[
\Delta_Y \circ f, (f\smsh f)\circ \Delta_X \: X \to Q_{\Bbb Z_2}(Y^{[2]})^{\Bbb Z_2}\, .
\]
Furthermore, it is easily checked that these maps are both coequalized by $\phi$:
\[
\xymatrix{
X \ar@<-.3ex>[rr]^(.3){\Delta_Y \circ f} \ar@<.3ex>[rr]_(.3){(f\smsh f)\circ \Delta_X} && Q_{\Bbb Z_2}(Y^{[2]})^{\Bbb Z_2} @>\phi >> Q(Y) \, .
}
\]
Set
\[
\{X,Y\smsh Y\}_{G}^{\Bbb Z_2} := [X, Q_{\Bbb Z_2}(Y^{[2]})^{\Bbb Z_2}]_G\, ,
\]
i.e., the abelian group of homotopy classes of $G$-equivariant maps $X\to Q_{\Bbb Z_2}(Y^{[2]})^{\Bbb Z_2}$.
Then  $\Delta_Y \circ f$ and $(f\smsh f)\circ \Delta_X$ represent elements of $\{X,Y\smsh Y\}_{G}^{\Bbb Z_2}$. 
Moreover, from \eqref{eqn:tomDieck} and the above discussion, one has a canonically split inclusion
\[
\iota\: \{X,D_2Y\}_G \to \{X,Y^{[2]}\}_{G}^{\Bbb Z_2} \, .
\]
It follows that on the level of homotopy classes, there is a unique $G$-equivariant stable homotopy class
\[
H(f) \in \{X,D_2(Y)\}_{G}
\]
such that 
\begin{equation} \label{eqn:hopfinv}
\iota\circ H(f) = (f\smsh f)\circ \Delta_X - \Delta_Y \circ f\, .
\end{equation}
Then $H(f)$ is the {\it  ($G$-equivariant geometric) Hopf invariant} of the stable $G$-map $f\: X\stableto Y$.

\begin{rem} In view of the fact that $\iota$ is a split injection, we will typically write
$H(f)$ in place of $\iota \circ H(f)$ in \eqref{eqn:hopfinv}.
\end{rem}

\begin{rem} \label{rem:G-pi} Crabb and Ranicki \cite[\S7]{Crabb-Ranicki} only consider the case when $G$ is a discrete group. However,
the adaptation to cofibrant topological groups $G$ goes through unchanged. 
In addition, we note that there is a $1$-connected Hurewicz homomorphism
$G\to \pi$, where $\pi= \pi_0(G)$. 
Under suitable restrictions on the dimension of $X$ and the connectivity of $Y$ (which are fulfilled in the cases we care about),
the evident homomorphism
\[
\{X,D_2(Y)\}_G \to  \{X\smsh_G (\pi_+),D_2(Y\smsh_G (\pi_+))\}_\pi
\]
which carries $H(f)$ to $H(f\smsh_G (\pi_+))$ is an isomorphism.  
Consequently, $H(f\smsh_G (\pi_+))$ will contain the same information as $H(f)$.
The reason we prefer to work with $G$ instead of $\pi$ is that it makes some of the proofs more transparent.
\end{rem}

\begin{defn} The {\it transfer} is the composition
\[
\sigma\:  \{X,D_2(Y)\}_{G} @> \iota >> \{X,Y^{[2]}\}_{G}^{\Bbb Z_2} \to \{X,Y^{[2]}\}_{G}
\]
in which the second homomorphism is induced by forgetting that a map is 
$\Bbb Z_2$-equivariant.
\end{defn}

\subsection{The diagonal embedding obstruction $\mu_M$}
In \cite{Klein_diagonals} the first author defines an obstruction to finding a Poincar\'e embedding of the diagonal as the reduced Hopf invariant
of the $(\pi\times \pi)$-equivariant stable map given by inducing $\Delta_!$ along
the homomorphism $G\times G \to \pi\times \pi$. After some minor rewriting, takes the form
\[
\Delta_! \smsh_{G\times G} ((\pi \times \pi)_+) \: \bar M_+ \smsh \bar M_+ \stableto \hat M^\tau \, ,
\]
where $\bar M \to M$ is the universal cover, $\hat M = \bar M \times_\pi (\pi \times \pi)$ and
$\hat M^\tau = S^\tau \smsh_{G} ((\pi \times \pi)_+)$. 
It is straightforward to check that the Hopf invariants in each case coincide (cf.~ Remark \ref{rem:G-pi}). Hence,
\[
H(\Delta_!) = H(\Delta_! \smsh_{G\times G}( (\pi \times \pi)_+)) \, .
\]
It immediately follows that 
\[
\mu_M = H(\Delta_!)\, ,
\]
where $\mu_M$ is the complete obstruction to finding a Poincar\'e embedding of the diagonal.

\begin{rem}
In the above, we have implicitly used Lemma \ref{lem:free-loop-ident}, Remark \ref{rem:free-loop-ident}  and the proof of Proposition \ref{prop:free-loop-ident}
 to define a preferred isomorphism identify $\Bbb Z\langle \bar \pi \rangle$ with 
\[
\Bbb Z\langle \bar \pi \rangle \cong \{EG_+\smsh EG_+, S^\tau[G^\lr] \smsh S^\tau[G^\lr] \}_{G} \, ,
\]
 as well as a preferred isomorphism
\[
Q_d(\pi) \cong \{EG_+\smsh EG_+, D_2(S^\tau[G^\lr])\}_{G} \, .
\] 
We refer the reader to \cite[\S6]{Klein_diagonals} for more details
(note that the discussion in \cite[\S6]{Klein_diagonals} carries over {\it mutatis mutandis} from $\pi$-spectra to  $G$-spectra).
\end{rem}

\section{Proof of the main results}

As noted in the introduction Theorem \ref{bigthm:rank-thm} follows immediately from Theorem \ref{bigthm:main}.

\begin{proof}[Proof of Theorem \ref{bigthm:main}]
From the definition of the Hopf invariant we immediately obtain
\[
\tr H(\Delta_!) = I(M) - \mathfrak e(M) \in H_0(LM).
\]
By Proposition \ref{prop:free-loop-ident}, $\tilde{\mathfrak e}(M) \in H_0(LM,M)$ is trivial, and by Proposition \ref{prop:I=r}
\[
I(M) = r(M).
\]
We thus conclude that
\[
\tilde \tr \mu_M = \tilde \tr H(\Delta_!) = \tilde r(M),
\]
showing the first assertion. The second part follows immediately from \cite[Theorem A]{Klein_diagonals}, which states that $\mu_M = 0$ if there exists a Poincar\'e embedding of the diagonal.
\end{proof}

\begin{proof}[Proof of Corollary \ref{bigcor:main}]
If $M$ is homotopy finite or if the Bass trace conjecture holds for $\pi$, then
$\tilde r(M) = 0$. The result now follows from Theorem \ref{bigthm:main}.
\end{proof}

\begin{proof}[Proof of Corollary \ref{bigcor:ex}] 
In this instance, the Bass trace conjecture holds for $\pi$, so $\tilde \tr(\mu_M) = \tilde r(M)= 0$.
The composition
\[
Q_d(\pi) @> \sigma >> \Bbb Z \langle \bar \pi \rangle @>>> Q_d(\pi) \, ,
\]
is given by multiplication by 2, where $\sigma$ is the transfer, and the second displayed map is the projection onto coinvariants.
By Theorem \ref{bigthm:main},
$\tilde \sigma(\mu_M) = 0$.  Consequently $2\mu_M =0$, so
by \cite[thm.~A]{Klein_diagonals}, it suffices to show that 
$Q_d(\pi) = \Bbb Z \oplus \tilde Q_d(\pi)$ has no 2-torsion elements (here
we have used the fact that $d$ is even to identify $\tilde Q_d(e) =\Bbb Z$).

Let use write $\pi$ as a direct product of $H \times K$ in which $H\subset \pi$ is the $2$-primary
subgroup.  Denote the involution  $x \mapsto x^{-1}$ of $\pi$ by $\iota$.
Then $\iota$ preserves the decomposition of $\pi$, so
$Q_d(\pi) \cong Q_d(H) \otimes Q_d(K)$. It then suffices to show that
$Q_d(H)$ and $Q_d(K)$ are torsion free abelian groups.

Since $d$ is even and
$\iota$ fixes $H$ pointwise, we have
$Q_d(H) = \Bbb Z\langle H \rangle$, so $Q_d(H)$ is torsion free. 

Set $K^* = K \setminus \{e\}$.
Then  $\iota\: K^*\to K^*$ is fixed point free. Consequently, there is a decomposition
\[
K^\ast =  K_- \amalg K_+
 \]
 in which $K_\pm \subset K^*$ are subsets, and $\iota\: K_- \to K_+$ is an isomorphism. It follows
 that $\tilde Q_d(K) \cong \Bbb Z\langle K_- \rangle$. We infer that $\tilde Q_d(K)$ is torsion free. Hence,
 $Q_d(K) = \Bbb Z \oplus \tilde Q_d(K)$ is also torsion free.
\end{proof}

\begin{proof}[Proof of Corollary \ref{bigcor:ex2}] The argument is similar
to the proof of Corollary \ref{bigcor:ex}.
The composition
\[
\tilde Q_d(\pi) @> \tilde \sigma >> \widetilde{\Bbb Z}\langle \bar \pi \rangle @>>> \tilde Q_d(\pi) \, ,
\]
is multiplication by $2$, where the second displayed map is the projection onto coinvariants.
By Theorem \ref{bigthm:main},
$\tilde \sigma(\mu_M) = 0$. Consequently, $2\mu_M = 0$ and it will suffice to show that
$\tilde Q_d(\pi) $ has no 2-torsion. We will see that $\tilde Q_d(\pi)$ is torsion free.

Let $\bar \pi^* =\bar \pi \setminus \{e\}$. Let $\iota\: \pi^\ast \to \pi^\ast$ be the involution given by
$\iota(x) = x^{-1}$.  We claim that $\iota$ defines a
induces a free action of $\Bbb Z_2$ on $\bar \pi^*$.  We prove this by contradiction: if not, then there 
are non-trivial elements $g,x\in \pi$ such that 
\begin{equation} \label{eqn:group}
gxg^{-1} = x^{-1}
\end{equation} 
 Inverting both sides \eqref{eqn:group}, 
we obtain $gx^{-1}g^{-1} = x$.   Substituting the last expression into \eqref{eqn:group}, we infer that 
$g^2x^{-1}g^{-2} = x^{-1}$. Inverting both sides of this last expression, we obtain $g^2xg^{-2} = x$.

 It follows that $g^2$ commutes with $x$. Let $2k+1$ be the order of $g$.
Then $g = (g^2)^{k+1}$ also commutes with $x$. Consequently,  $gxg^{-1} = x$. But by assumption
$gxg^{-1} = x^{-1}$, so
$x = x^{-1}$. Hence, $x^2 = e$ and we obtain a contradiction.  

It follows that $\bar \pi^* = T_- \amalg T_+$ where $\iota$ restricts to
an isomorphism $T_- \to T_+$.  Hence,
\[
\tilde Q_d(\pi) \cong \Bbb Z\langle T_-\rangle
\]
is free abelian with basis $T_-$.  
\end{proof}

\bibliographystyle{elsarticle-num}
\bibliography{john}


\end{document}